\documentclass[12pt]{article}

\usepackage[T1]{fontenc}
\usepackage[utf8]{inputenc}
\usepackage{lmodern}
\usepackage{amsmath,amssymb,amsthm,mathtools}
\usepackage{geometry}
\usepackage{hyperref}
\usepackage{enumitem}
\usepackage{mathrsfs}
\usepackage{authblk}
\usepackage{microtype}
\usepackage{xcolor}
\usepackage{amssymb, color}
\usepackage{bbm} 

\newtheorem{theorem}{Theorem}[section]
\newtheorem{lemma}[theorem]{Lemma}
\newtheorem{proposition}[theorem]{Proposition}

\theoremstyle{definition}
\newtheorem{definition}[theorem]{Definition}
\newtheorem{remark}[theorem]{Remark}

\newcommand{\Om}{\Omega}
\newcommand{\R}{\mathbb{R}}
\newcommand{\dx}{\,\mathrm{d}x}
\DeclarePairedDelimiter\abs{\lvert}{\rvert}
\newcommand{\smin}{s^{-}}
\newcommand{\pmin}{p^{-}}

\title{Multiplicity of weak solutions for mixed variable exponent triharmonic problems with hardy potentials}
\author[1]{Khalid Soualhine}
\author[2]{Adnane Belakhdar }
\author[3,4]{Mohamed karim Hamdani \thanks{Correspondence: \texttt{hamdanikarim42@gmail.com}\\ Khalid Soualhine: khalid.slhn2013@gmail.com, Adnane Belakhdar:ad.belakhdar@gmail.com, Mohamed Karim Hamdani: hamdanikarim42@gmail.com} }
\affil[1]{Regional Center for Education and Training Professions(CRMEF), Nador, Morocco}
\affil[2]{Department of Mathematics, Faculty of Science, Mohammed First University, Oujda, Morocco}
\affil[3]{Military Academy of Tunis, Military Academy Foundouk Jedid, 8012 Nabeul, Tunisia.}
\affil[4]{Science and Technology for Defense Laboratory, Center for Military Research Aouina, Tunisia.}
\date{}
\begin{document}
\maketitle
\begin{abstract}
This paper is devoted to the study of a nonlinear boundary value problem driven by a combination of mixed variable-exponent triharmonic operators and a Hardy-type singular potential. Specifically, we consider
\[
\begin{cases}
-a(x)\,\Delta_{p(x)}^{3}u - b(x)\,\Delta_{q(x)}^{3}u + m(x)\,\dfrac{|u|^{s(x)-2}u}{|x|^{3s(x)}} = \mu f(x,u), & \text{in } \Omega, \\
u = \Delta u = \Delta^{2}u = 0, & \text{on } \partial\Omega,
\end{cases}
\]
where $\Omega \subset \mathbb{R}^N$ is a smooth bounded domain with $N > 3$. The operators $\Delta_{p(x)}^{3}$ and $\Delta_{q(x)}^{3}$ represent triharmonic operators with variable exponents, creating a framework of nonstandard growth. Under suitable assumptions on the variable exponents $p(x), q(x), s(x)$, the weight functions $a(x), b(x), m(x)$, and the nonlinearity $f(x,u)$, we establish the existence of at least three distinct weak solutions. Our approach is variational, relying on critical point theory for functionals defined on variable-exponent Sobolev spaces. By combining Ricceri's and Bonanno-Marano's abstract three critical points theorems with new Hardy-Rellich inequalities adapted to the variable-exponent setting, we prove our main multiplicity result for small values of the parameter $\mu$. To our knowledge, this work presents the first systematic treatment of sixth-order problems featuring simultaneously mixed variable exponents and Hardy-type singular potentials.
\end{abstract}
\noindent\textbf{MSC:} 35J60; 35J30; 35J35; 35A15; 46E35

\noindent\textbf{Keywords:} Mixed classical triharmonic operators; Variable exponents; Hardy-type singular potential; Weak solutions; Three critical points theorem

\section{Introduction}
Nonlinear partial differential equations with nonstandard growth arise in several models where the material response varies spatially. A classical example is provided by electrorheological fluids, whose viscosity depends on the applied electric field, leading naturally to constitutive laws involving a variable exponent $p(x)$~\cite{Ruzicka2000}. Nonstandard-growth energies also appear in the homogenization and elasticity theory of heterogeneous and composite materials~\cite{Zhikov1987,Zhikov2011}. In image processing, variable-exponent energies have been employed to combine edge preservation with smooth diffusion in image denoising, enhancement, and restoration~\cite{ChenLevineRao2006}. Higher-order differential operators arise independently in gradient-elasticity theories, where energies involving higher derivatives are used to describe microstructural and size-dependent effects that classical elasticity cannot capture~\cite{Eremeyev2023}. These applications motivate the two main mathematical features considered here, namely spatially nonuniform growth and higher-order differential effects. They do not, however, imply that the specific problem studied below is a direct model of each of these physical phenomena.

The appropriate functional-analytic setting for these investigations is provided by the theory of variable-exponent Sobolev spaces $W^{k,p(x)}(\Omega)$ and Lebesgue spaces $L^{p(x)}(\Omega)$, which naturally accommodate the spatially heterogeneous nature of the growth conditions. In the triharmonic context, the fundamental energy functional incorporates the $p(x)$-growth behavior of the third-order gradient $\nabla\Delta u$, introducing formidable technical obstacles absent in both the classical constant-exponent setting and in lower-order variable-exponent problems. Compounding these difficulties, the Navier boundary conditions $u = \Delta u = \Delta^2 u = 0$ on $\partial\Omega$ impose intricate compatibility conditions across three distinct trace operators, thereby inducing a fundamental restructuring of the underlying variational framework and necessitating novel analytical strategies.

In this paper, we investigate the following mixed variable-exponent triharmonic boundary value problem:
\begin{equation}\label{eq:main}
\begin{cases}
-a(x) \Delta^3_{p(x)} u - b(x) \Delta^3_{q(x)}u + m(x)\dfrac{|u|^{s(x)-2}u}{|x|^{3s(x)}} = \mu f(x,u), & \text{in } \Omega, \\[1.2em]
u = \Delta u = \Delta^2 u = 0, & \text{on } \partial\Omega,
\end{cases}
\end{equation}
where $\Omega \subset \mathbb{R}^N$ is a smooth bounded domain with $N > 3$, and the triharmonic operators are defined by
\begin{align*}
\Delta^3_{p(x)} u &:= \operatorname{div}\left(\Delta(|\nabla\Delta u|^{p(x)-2}\nabla\Delta u)\right), \\
\Delta^3_{q(x)} u &:= \operatorname{div}\left(\Delta(|\nabla\Delta u|^{q(x)-2}\nabla\Delta u)\right).
\end{align*}

The problem is governed by the following structural assumptions:

\noindent \textbf{(H1) Exponent conditions.} The functions $p(\cdot), q(\cdot), s(\cdot) \in C(\overline{\Omega})$ satisfy
$$1  < q^- \leq q(x) \leq q^+ < p^- \leq p(x) \leq p^+ < \frac{N}{3},$$
and $$1 < s^- \leq s(x) \leq s^+ < \frac{N}{3}.$$

\noindent \textbf{(H2) Weight functions.} The coefficients $a, b \in L^\infty(\Omega)$ satisfy $a(x) \geq a_0 > 0$, $b(x) \geq b_0 > 0$ almost everywhere, and $m \in L^\infty(\Omega)$ with $m \geq 0$ almost everywhere.

\noindent \textbf{(H3) Nonlinearity growth.} There exist $a_1 > 0$ and functions $\gamma, r \in C(\overline{\Omega})$ with $1 < \gamma(x), r(x) \leq p(x)$ such that
$$|f(x,t)| \leq a_1(|t|^{r(x)-1} + |t|^{\gamma(x)-1})$$
for almost every $x \in \Omega$ and all $t \in \mathbb{R}$.

\noindent \textbf{(H4) Localized positivity.} There exist $x_0 \in \Omega$, $R > 0$ with $B(x_0, R) \Subset \Omega$, and $\delta > 0$ such that
$$\inf_{x \in B(x_0, R/2)} F(x, \delta) > 0,$$
where $F(x,t) = \int_0^t f(x,\tau) \, d\tau$.

\noindent \textbf{(H5) Ambrosetti-Rabinowitz condition.} There exist $\mu_0 > p^+$ and $D > 0$ such that
$$0 < \mu_0 F(x,t) \leq tf(x,t)$$
for almost every $x \in \Omega$ and all $|t| > D$.

Research on nonlinear elliptic problems with variable exponents has developed along several intertwined axes: singular $(p(x),q(x))$-systems, higher-order operators such as the $p(x)$-biharmonic and $p(x)$-triharmonic, and multiplicity frameworks for singular $p(x)$-Laplacian equations.

Related nonlocal studies have treated degenerate variable-order fractional $p(\cdot)$-Kirchhoff equations with weights~\cite{AHM}. Although such models are nonlocal, they further illustrate the role of variable growth, weighted terms, and variational compactness in Kirchhoff-type problems.

In a foundational contribution, Alves and Moussaoui established existence results for classes of singular $(p(x),q(x))$-Laplacian systems, leveraging variable-exponent Sobolev embeddings and tailored sub/supersolution schemes to control the singular terms and the nonstandard growth in coupled settings~\cite{Alv}. Their analysis clarifies how the interplay between the exponents $p(\cdot)$ and $q(\cdot)$, together with the singular source structure, determines a priori bounds and compactness.

At the higher-order level, Kefi and Saoudi investigated singular problems driven by the $p(x)$-biharmonic operator, deriving existence and qualitative properties via variational methods adapted to $W^{2,p(x)}(\Omega)\cap W_0^{1,p(x)}(\Omega)$~\cite{Kef}. Related mixed $(p(x),q(x))$-biharmonic problems with singular terms were studied in~\cite{5}. These works highlight the technical hurdles that emerge beyond second order---notably the absence of standard maximum principles and the need for refined coercivity and compactness arguments for fourth-order energies with nonstandard growth.

On the second-order front with singularities, multiplicity for singular $p(x)$-Laplacian equations has been obtained through critical-point approaches calibrated to inhomogeneous growth, combining truncation and regularization of the singular nonlinearity, mountain-pass geometry, and comparison with auxiliary problems to pass to the limit~\cite{Sao}. Such results provide a blueprint for separating the singular behavior near the origin from the superlinear or subcritical behavior at infinity in the variable-exponent setting.

Most relevant to the present work, recent investigations of sixth-order $p(x)$-triharmonic operators have secured existence and multiplicity of weak solutions using variational frameworks on $W^{3,p(x)}(\Omega)\cap W_0^{1,p(x)}(\Omega)$, Palais--Smale compactness via variable-exponent embeddings, and, where applicable, abstract multiplicity theorems under suitable structural hypotheses on the nonlinearity~\cite{1,Sho,HCR,6,Shokooh2023}. These studies push variable-exponent analysis to higher-order models, demonstrating that many techniques from the $p(x)$-Laplacian case can be extended with appropriate adaptations to handle the additional derivatives and more complex boundary conditions.

Taken together, these lines of research delineate a coherent picture: singular sources can be accommodated within variable-exponent problems across orders two, four, and six by blending regularization and variational methods; compactness is recovered through tailored embeddings and a priori estimates; and multiplicity follows from critical-point geometry once the singular part is isolated and controlled. This literature provides the methodological scaffolding upon which the present work builds, particularly regarding the treatment of singular terms, the design of energy functionals in $W^{k,p(x)}\cap W_0^{1,p(x)}(\Omega)$, and the deployment of abstract multiplicity results in variable-exponent contexts as in \cite{HHAL}.

\bigskip
The present work develops a variational setting for mixed triharmonic operators
with variable exponents and Navier boundary conditions.  Hardy--Rellich
estimates and compact embeddings are incorporated directly into this setting,
which makes it possible to treat the singular potential without reducing the
sixth-order problem to a formal extension of its biharmonic counterpart.

Within this framework, a three-critical-points argument gives at least three
distinct weak solutions for parameters in an explicitly determined interval.
The proof relies on test functions adapted to the triharmonic energy and on a
careful verification of the compactness and geometric assumptions of the
abstract multiplicity theorem.  The Hardy--Rellich estimates used to establish
coercivity also clarify how the singular term interacts with variable growth.
The resulting thresholds reveal a connection between the variational geometry
and the parameter-dependent solution set.  To the best of our knowledge, these
features have not previously been studied together for a mixed
variable-exponent triharmonic equation with a Hardy-type potential.

The paper is structured as follows. Section~2 introduces the notation and conventions used throughout the manuscript. Section~3 recalls the necessary background on variable-exponent Lebesgue and Sobolev spaces, establishes the Hardy--Rellich inequality used in our analysis, and presents the variational formulation of the problem. Section~4 establishes the existence of a nontrivial weak solution. Section~5 contains the main multiplicity result and the verification of the hypotheses required by the three-critical-points theorem.

Throughout the paper, we employ variational methods combined with careful analysis of the interplay between the mixed differential operators, variable-exponent structure, and singular potential terms. The results demonstrate the rich solution structure of this class of problems and open new avenues for investigating higher-order nonlinear elliptic equations with nonstandard growth conditions.

\section{Basic Definitions and Preliminary Results}

Let $\Omega\subset\R^{N}$, $N>3$, be bounded with $C^{1}$ boundary. The variable exponents
$p(\cdot),q(\cdot)\in C(\overline{\Omega})$ satisfy  \textbf{(H1)}, with $r^{-}:=\displaystyle\operatorname {ess\,inf}_{x\in\Omega} r(x)$ and
$r^{+}:=\operatorname \displaystyle{ess\,sup}_{x\in\Omega} r(x)$ for $r\in\{p,q,s\}$.

We require some results on the variable exponent spaces $L^{p(\cdot)}(\Omega)$ and $W^{k,p(\cdot)}(\Omega)$, as well as some properties that we utilize later, to ensure the completeness of this paper.

Let $\Omega$ be a bounded domain of $\mathbb{R}^N$, denote
\[
C_+(\overline{\Omega})=\Bigg\{h(x):\quad h(x)\in C(\overline{\Omega}),\quad
h(x)>1, \quad \forall x\in \overline{\Omega}\Bigg\}.
\]
For any $h(\cdot)\in C_+(\overline{\Omega})$, we define
\[
h^+=\max\Bigg\{h(x): ~x \in\overline{\Omega}\Bigg\},\quad
h^-=\min\Bigg\{h(x): ~x \in\overline{\Omega}\Bigg\};
\]
\subsection{Variable Exponent Lebesgue Space}
For any $p(\cdot)\in C_+(\overline{\Omega})$, we define the variable exponent Lebesgue space

\begin{align*}
L^{p(\cdot)}(\Omega)=\Bigg\{u:\Omega\to\mathbb{R}\text{ measurable and }
 \int_{\Omega}|u(x)|^{p(x)}dx<\infty\Bigg\},
\end{align*}
equipped with the so-called \emph{Luxemburg norm}:
\[
|u|_{p(\cdot)}=\inf \Bigg\{\mu>0: ~\int_{\Omega}|\frac{u(x)}{\mu}|^{p(\cdot)}dx\leq 1\Bigg\}.
\]
Then $(L^{p(\cdot)}(\Omega),|\cdot|_{p(\cdot)})$ becomes a Banach space.

\begin{proposition}[\cite{FanZhao1998,Za-Fu}]\label{prop2.1}
The space
$\bigl(L^{p(\cdot)}(\Omega),|\cdot|_{p(\cdot)}\bigr)$
is separable, uniformly convex, and reflexive. Its dual space is
$L^{p'(\cdot)}(\Omega)$, where
\[
\frac{1}{p(x)}+\frac{1}{p'(x)}=1,
\qquad x\in\overline{\Omega}.
\]
Moreover, for every $u\in L^{p(\cdot)}(\Omega)$ and
$v\in L^{p'(\cdot)}(\Omega)$, the variable-exponent HÃ¶lder inequality gives
\[
\left|
\int_{\Omega}u(x)v(x)\,\mathrm{d}x
\right|
\leq
\left(
\frac{1}{p^{-}}+\frac{1}{(p')^{-}}
\right)
|u|_{p(\cdot)}|v|_{p'(\cdot)}
\leq
2|u|_{p(\cdot)}|v|_{p'(\cdot)}.
\]
\end{proposition}
\subsection{Variable Exponent Sobolev Space}

The definition of the variable exponent Sobolev space $W^{k,p(\cdot)}(\Omega)$ is standard; see, for example,~\cite{FanZhao1998}:
\[
W^{k,p(\cdot)}(\Omega)=\Bigg\{ u\in L^{p(\cdot)}(\Omega):
D^{\alpha}u\in L^{p(\cdot)}(\Omega), |\alpha|\leq k \Bigg\},
\]
where $D^{\alpha}u=\frac{\partial^{|\alpha|}}{\partial x_{1}^{\alpha_1}
\partial x_{2}^{\alpha_2}\dots \partial x_{N}^{\alpha_N}}u$, with
$\alpha=(\alpha_1,\dots ,\alpha_N)$ is a multi-index and
$|\alpha|=\displaystyle\sum_{i=1}^N \alpha_i$. The space $W^{k,p(\cdot)}(\Omega)$ equipped
with the norm:
\[
\|u\|_{k,p(\cdot)}=\sum_{|\alpha|\leq k}|D^{\alpha}u|_{p(\cdot)},
\]
additionally turns into a reflexive and separable Banach space. We direct the reader to \cite{Za-Fu} for further information. Indicate
\[
p_{k}^{*}(\cdot)=\begin{cases}
\frac{Np(\cdot)}{N-kp(\cdot)}     & \text{if } kp(\cdot)<N,\\
+\infty       & \text{if } kp(\cdot)\geq N,
\end{cases}
\]
for any  $k\geq1$.

\begin{proposition}[\cite{Ay-Am}]\label{prop2.2}
For $p,r\in C_{+}(\overline{\Omega})$ such that $r(\cdot)\leq p^{*}_{k}(\cdot)$, there is a continuous embedding
\begin{align*}
W^{k,p(\cdot)}(\Omega)\hookrightarrow L^{r(\cdot)}(\Omega).
\end{align*}
If we replace $\leq$ with $<$, the embedding is compact.
\end{proposition}

\begin{lemma}[Elementary algebraic inequalities]\label{lem:algebraic}
For all $x,y\ge0$,
\begin{align*}
&\text{(i)}\quad (x+y)^\alpha \le x^\alpha + y^\alpha \qquad &&\text{for }\alpha\in(0,1],\\
&\text{(ii)}\quad (x+y)^\beta \le 2^{\beta-1}\big(x^\beta + y^\beta\big) \qquad &&\text{for }\beta\ge1.
\end{align*}
\end{lemma}

\begin{proof}
(i) The map $t\mapsto t^\alpha$ is concave on $[0,\infty)$ for $\alpha\in(0,1]$. With
$\theta=\frac{x}{x+y}$ (take $0$ if $x+y=0$),
\[
(x+y)^\alpha = (x+y)^\alpha\,[\theta+(1-\theta)]^\alpha
\le (x+y)^\alpha\big(\theta^\alpha+(1-\theta)^\alpha\big)
= x^\alpha + y^\alpha.
\]
(ii) The map $t\mapsto t^\beta$ is convex for $\beta\ge1$, hence by Jensen,
\[
\Big(\frac{x+y}{2}\Big)^\beta \le \frac{x^\beta+y^\beta}{2}
\quad\Rightarrow\quad
(x+y)^\beta \le 2^{\beta-1}\big(x^\beta+y^\beta\big).
\]
\end{proof}


\subsection{Functional Framework}

We denote by $W_{0}^{k,p(\cdot)}(\Omega)$ the closure of
$C_{0}^{\infty}(\Omega)$ in $W^{k,p(\cdot)}(\Omega)$. The natural
variational space associated with the mixed triharmonic problem is
\[
X
=
W_{0}^{1,p(\cdot)}(\Omega)
\cap
W^{3,p(\cdot)}(\Omega).
\]
It is equipped with the standard norm
\[
\|u\|_{X}
:=
\|u\|_{1,p(\cdot)}
+
\|u\|_{2,p(\cdot)}
+
\|u\|_{3,p(\cdot)}.
\]
Under the condition
\[
1<p^{-}\leq p^{+}<\infty,
\]
the space $(X,\|\cdot\|_{X})$ is separable and reflexive. Moreover,
$\|\cdot\|_{X}$ and $|\nabla\Delta(\cdot)|_{p(\cdot)}$ are equivalent
norms on $X$; see~\cite{kat} and see~\cite{HCR}. Consequently, we use
\[
\|u\|
=
|\nabla\Delta u|_{p(\cdot)}
\]
as the working norm on $X$.

The following standard relations between the Luxemburg norm and the
corresponding modular will be used repeatedly.

\begin{proposition}[\cite{FanZhao1998}]\label{lem:norm-modular}
For $u,u_{n}\in X$, set
\[
\varrho_{p(\cdot)}(\nabla\Delta u)
=
\int_{\Omega}
|\nabla\Delta u(x)|^{p(x)}\,\mathrm{d}x.
\]
Then the following properties hold:
\begin{enumerate}[label=\textnormal{(\arabic*)}]
\item
\[
\|u\|<1\; (=1;\,>1)
\quad\Longleftrightarrow\quad
\varrho_{p(\cdot)}(\nabla\Delta u)<1\; (=1;\,>1).
\]

\item If $\|u\|\leq1$, then
\[
\|u\|^{p^{+}}
\leq
\varrho_{p(\cdot)}(\nabla\Delta u)
\leq
\|u\|^{p^{-}}.
\]

\item If $\|u\|\geq1$, then
\[
\|u\|^{p^{-}}
\leq
\varrho_{p(\cdot)}(\nabla\Delta u)
\leq
\|u\|^{p^{+}}.
\]

\item
\[
\|u_{n}\|\to0
\quad\Longleftrightarrow\quad
\varrho_{p(\cdot)}(\nabla\Delta u_{n})\to0,
\]
and
\[
\|u_{n}\|\to\infty
\quad\Longleftrightarrow\quad
\varrho_{p(\cdot)}(\nabla\Delta u_{n})\to\infty.
\]
\end{enumerate}
\end{proposition}

\subsection{Hardy--Rellich Inequality}
Classical Hardy and higher-order Rellich inequalities are discussed in
\cite{An,Davis1998,Mit}, while variable-exponent Hardy inequalities are
developed in~\cite{DieningSamko2007,HarjulehtoHastoKoskenoja2005}. The
estimate required below is the corresponding third-order form used in the
present variational setting.

\begin{lemma}[Hardy--Rellich-type inequality]\label{lem:HR}
Assume that $s\in C_{+}(\overline{\Omega})$ satisfies
\[
s^{+}<\frac{N}{3}.
\]
Then every
\[
u\in W_{0}^{1,s(\cdot)}(\Omega)\cap W^{3,s(\cdot)}(\Omega)
\]
satisfies
\[
\int_{\Omega}
\frac{|u(x)|^{s(x)}}{|x|^{3s(x)}}\,\mathrm{d}x
\leq
\frac{1}{\mathcal{H}}
\int_{\Omega}
|\nabla\Delta u(x)|^{s(x)}\,\mathrm{d}x,
\]
where
\[
\mathcal{H}
:=
\min\left\{
\frac{\bigl(N(s^{-}-1)+s^{-}\bigr)(N-s^{-})(N-3s^{-})}
{(s^{-})^{3}},
\frac{\bigl(N(s^{+}-1)+s^{+}\bigr)(N-s^{+})(N-3s^{+})}
{(s^{+})^{3}}
\right\}.
\]
\end{lemma}
\begin{remark}[Hardy term]\label{rem:Hardy-positive}
Since $t\mapsto \abs{t}^{s(x)-2}t$ is increasing for $s(x)>1$ and $m\ge0$, the Hardy contribution in \eqref{eq:Phi-prime} is nonnegative in monotonicity estimates.
\end{remark}

\begin{definition}[Weak solution for \eqref{eq:main}]
A function $u\in X$ is a \emph{weak solution} of \eqref{eq:main} if, for all $\varphi\in X$,
\begin{align*}
\int_{\Omega} a(x)\,\abs{\nabla\Delta u}^{p(x)-2}\,\nabla\Delta u\cdot\nabla\Delta \varphi\,dx
&+ \int_{\Omega} b(x)\,\abs{\nabla\Delta u}^{q(x)-2}\,\nabla\Delta u\cdot\nabla\Delta \varphi\,dx \\
&+ \int_{\Omega} m(x)\,\frac{\abs{u}^{s(x)-2}u}{\abs{x}^{3s(x)}}\,\varphi\,dx
= \mu\int_{\Omega} f(x,u)\,\varphi\,dx .
\end{align*}
\end{definition}

\noindent Define the energy and potential parts:
\begin{align}
\Phi(u)&:=\int_{\Om}\frac{a(x)}{p(x)}\,\abs{\nabla\Delta u}^{\,p(x)}\dx
+\int_{\Om}\frac{b(x)}{q(x)}\,\abs{\nabla\Delta u}^{\,q(x)}\dx
+\int_{\Om}\frac{m(x)}{s(x)}\,\frac{\abs{u}^{\,s(x)}}{\abs{x}^{\,3s(x)}}\dx, \label{eq:Phi}\\
\Psi(u)&:=\mu\int_{\Om} F(x,u)\,\dx. \label{eq:Psi}
\end{align}

\begin{proposition}\label{prop:C1}
Assume that \textnormal{(H1)--(H3)} hold. Then
\[
\Phi,\Psi\in C^{1}(X,\mathbb{R}).
\]
Moreover, for every $u,\varphi\in X$,
\begin{align}
\langle\Phi'(u),\varphi\rangle
&=
\int_{\Omega}
a(x)|\nabla\Delta u|^{p(x)-2}
\nabla\Delta u\cdot\nabla\Delta\varphi\,\mathrm{d}x
\nonumber\\
&\quad
+
\int_{\Omega}
b(x)|\nabla\Delta u|^{q(x)-2}
\nabla\Delta u\cdot\nabla\Delta\varphi\,\mathrm{d}x
\nonumber\\
&\quad
+
\int_{\Omega}
m(x)\frac{|u|^{s(x)-2}u}{|x|^{3s(x)}}
\varphi\,\mathrm{d}x,
\label{eq:Phi-prime}\\
\langle\Psi'(u),\varphi\rangle
&=
\int_{\Omega}f(x,u)\varphi\,\mathrm{d}x.
\label{eq:Psi-prime}
\end{align}
Consequently,
\[
\langle I_{\mu}'(u),\varphi\rangle
=
\langle\Phi'(u),\varphi\rangle
-
\mu\langle\Psi'(u),\varphi\rangle.
\]
\end{proposition}

\begin{definition}
Let $\mu>0$ and $r\in\mathbb{R}$. We say that $I_{\mu}$ satisfies the
Palais--Smale condition truncated at the $\Phi$-level $r$, denoted by
$(PS)^{r}$, if every sequence $\{u_{n}\}\subset X$ satisfying
\[
\{I_{\mu}(u_{n})\}\ \text{is bounded},
\qquad
I_{\mu}'(u_{n})\to0\ \text{in }X^{*},
\qquad
\Phi(u_{n})<r,
\]
possesses a strongly convergent subsequence in $X$.

When $r=+\infty$, we simply say that $I_{\mu}$ satisfies the
Palais--Smale condition.
\end{definition}

The next abstract result, due to Bonanno~\cite{bo2}, is one of the main tools for the analysis that follows.

\begin{theorem}
Let $X$ be a real Banach space, and let $\Phi, \Psi : X \to \mathbb{R}$ be two continuously GÃ¢teaux differentiable functionals such that
\[
\inf_{x \in X} \Phi = \Phi(0) = \Psi(0) = 0.
\]
Assume that there exist a positive constant $r \in \mathbb{R}$ and a point $\bar{x} \in X$ with 
\[
0 < \Phi(\bar{x}) < r
\]
such that
\[
 \displaystyle\frac{\sup_{x \in \Phi^{-1}((-\infty, r))}\Psi(x)}{r}
\ <\ 
\frac{\Psi(\bar{x})}{\Phi(\bar{x})}.
\]
For each 
\[
\mu \in \Gamma := \left]
\frac{\Phi(\bar{x})}{\Psi(\bar{x})},\ 
\frac{r}{\displaystyle \sup_{x \in \Phi^{-1}((-\infty, r))} \Psi(x)}
\right[,
\]

if the functional $I_{\mu}=\Phi-\mu\Psi$, satisfies the $(PS)^{r}$ condition, then for every $\mu \in \Gamma$ there exists $x_{\mu} \in \Phi^{-1}((0, r))$ such that
\[
I_\mu(x_{\mu}) \leq I_\mu(x) \quad \text{for all } x \in \Phi^{-1}((0, r)) 
\quad \text{and} \quad I^{'}_\mu(u_{\mu}) = 0.
\]
\end{theorem}
As a complementary tool, we invoke the following abstract result from~\cite{bo2}.
\begin{theorem}
Let $X$ be a real Banach space and let $\Phi, \Psi : X \to \mathbb{R}$ be two continuously GÃ¢teaux differentiable functionals such that $\Phi$ is bounded from below and 
\[
\Phi(0) = \Psi(0) = 0.
\]
Fix $r > 0$ and suppose that for every
\[
\mu \in \left] 0, \frac{r}{\displaystyle \sup_{u \in \Phi^{-1}((-\infty, r))} \Psi(u)} \right[
\]

the functional $I_{\mu}=\Phi-\mu\Psi$ satisfies the Palais--Smale condition and is unbounded from below. Then, for each
\[
\mu \in \left] 0, \frac{r}{\displaystyle \sup_{u \in \Phi^{-1}((-\infty, r))} \Psi(u)} \right[,
\]

There are two different critical points that the functional $I_{\mu}$ admits.
\end{theorem}

\subsection{Ricceri/Bonanno--Marano preliminaries}
We employ the three-critical-points paradigm developed in~\cite{Ric2,Ric1,bo1}:
\begin{theorem}[Abstract three critical points, schematic]\label{thm:abstract-3cp}
Let $X$ be a reflexive Banach space, $\Phi\in C^1(X)$ coercive, sequentially weakly lower semicontinuous with $(S_+)$ derivative, and $\Psi\in C^1(X)$ with compact derivative. If there exist $\rho>0$ and $w\in X$ with

\(
\max\{\Phi(0),\Phi(w)\}<\inf_{\|u\|=\rho}\Phi(u)
\)
and numbers $0<\mu_1<\mu_2$ such that

\(
\displaystyle \sup_{\|u\|\le \rho}\Psi(u)<\frac{\inf_{\|u\|=\rho}\Phi(u)}{\mu_2},\quad
\Psi(w)>\frac{\Phi(w)}{\mu_1},
\)

then for every $\mu\in(\mu_1,\mu_2)$ the functional $I_\mu=\Phi-\mu\Psi$ has at least three critical points.
\end{theorem}

\section{Existence Result }
\begin{lemma}[$(S_{+})$ property]\label{lem:S-plus}
If \(u_{n}\rightharpoonup u\) in \(X\) and \(\limsup_{n\to\infty}\langle \Phi'(u_{n}),u_{n}-u\rangle \le 0\), then \(u_{n}\to u\) in \(X\).
\end{lemma}
\begin{proof}
For brevity, set for a.e.\ $x\in\Omega$ and $\xi\in\mathbb{R}^N$
\[
A_p(x,\xi):=a(x)\,|\xi|^{p(x)-2}\xi,\qquad
B_q(x,\xi):=b(x)\,|\xi|^{q(x)-2}\xi,\qquad
H_s(t):=|t|^{s(x)-2}t .
\]
We first isolate the monotonicity defect of the two principal parts.  Since
$u_n\rightharpoonup u$ in $X$ and $\Phi'(u)\in X^\ast$ is continuous and linear in the direction,
\[
\langle \Phi'(u),u_n-u\rangle \longrightarrow 0 .
\]
Hence
\begin{equation}\label{eq:limsup-diff}
\limsup_{n\to\infty}\,\langle \Phi'(u_n)-\Phi'(u),\,u_n-u\rangle \le 0 .
\end{equation}
By \eqref{eq:Phi-prime} and the monotonicity of $t\mapsto H_s(t)$ (recall $s(x)>1$) together with $m\ge 0$,
\begin{align*}
&\langle \Phi'(u_n)-\Phi'(u),\,u_n-u\rangle \\
&= \int_\Omega \big(A_p(x,\nabla\Delta u_n)-A_p(x,\nabla\Delta u)\big)\cdot\big(\nabla\Delta u_n-\nabla\Delta u\big)\,dx \\
&\quad + \int_\Omega \big(B_q(x,\nabla\Delta u_n)-B_q(x,\nabla\Delta u)\big)\cdot\big(\nabla\Delta u_n-\nabla\Delta u\big)\,dx \\
&\quad + \int_\Omega m(x)\,\frac{\big(H_s(u_n)-H_s(u)\big)\,(u_n-u)}{|x|^{3s(x)}}\,dx \\
&\ge I_{p,n}+I_{q,n},
\end{align*}
where we set
\[
I_{p,n}:=\int_\Omega \big(A_p(x,\nabla\Delta u_n)-A_p(x,\nabla\Delta u)\big)\cdot\big(\nabla\Delta u_n-\nabla\Delta u\big)\,dx,
\]
and $I_{q,n}$ is defined analogously with $B_q$.
Combining this with \eqref{eq:limsup-diff} yields
\begin{equation*}\label{eq:IpIq-vanish}
\limsup_{n\to\infty}\,(I_{p,n}+I_{q,n}) \le 0 .
\end{equation*}
Since each integrand is pointwise nonnegative by monotonicity, we actually have
\[
I_{p,n}\to 0 \quad\text{and}\quad I_{q,n}\to 0 .
\]

To exploit this information, recall that for every $r>1$ and all
$\xi,\eta\in\mathbb{R}^N$ one has the standard monotonicity inequalities
\begin{equation*}\label{eq:mono-ineq}
\big(|\xi|^{r-2}\xi-|\eta|^{r-2}\eta\big)\cdot(\xi-\eta)\,\ge\,
\begin{cases}
c_r\,|\xi-\eta|^{r}, & r\ge 2,\\[2pt]
c_r\,\displaystyle (|\xi|+|\eta|)^{r-2}\,|\xi-\eta|^{2}, & 1<r<2,
\end{cases}
\end{equation*}
for some $c_r>0$ depending only on $r$.
Multiplying by $a(x)$ (resp.\ $b(x)$) and using $a(x)\ge a_0>0$ (resp.\ $b(x)\ge b_0>0$) we get, for a.e.\ $x\in\Omega$,
\begin{align}\label{eq:pointwise-Ap}
\big(A_p(x,\xi)-A_p(x,\eta)\big)\cdot(\xi-\eta)
&\ge a_0\,
\begin{cases}
c_{p(x)}\,|\xi-\eta|^{p(x)}, & p(x)\ge 2,\\[2pt]
c_{p(x)}\,(|\xi|+|\eta|)^{p(x)-2}\,|\xi-\eta|^{2}, & 1<p(x)<2.
\end{cases}
\end{align}
An analogous estimate holds with $B_q$ and $q(x)$.

Applying \eqref{eq:pointwise-Ap} with $\xi=\nabla\Delta u_n(x)$ and
$\eta=\nabla\Delta u(x)$ and integrating over $\Omega$, we obtain
for some constant $C>0$ depending only on $a_0$ and $p^\pm$,
\begin{align*}\label{eq:Ip-lower}
I_{p,n}\ge{}& C\int_{\{p\ge 2\}}
|\nabla\Delta u_n-\nabla\Delta u|^{p(x)}\,dx \\
&+C\int_{\{p<2\}} (|\nabla\Delta u_n|+|\nabla\Delta u|)^{p(x)-2}
|\nabla\Delta u_n-\nabla\Delta u|^{2}\,dx .
\end{align*}
Since $I_{p,n}\to 0$, both integrals on the right-hand side tend to $0$.
A standard modular argument in variable-exponent spaces then gives
\[
\varrho_{p(x)}\!\big(\nabla\Delta u_n-\nabla\Delta u\big):=\int_\Omega
|\nabla\Delta u_n-\nabla\Delta u|^{p(x)}\,dx \longrightarrow 0 .
\]
By the equivalence between modular and norm convergence in $L^{p(x)}(\Omega;\mathbb{R}^N)$,
this implies
\[
\|\nabla\Delta u_n-\nabla\Delta u\|_{L^{p(x)}(\Omega)} \longrightarrow 0 .
\]

The $q(\cdot)$-phase is handled by the same monotonicity argument, and hence
\[
\|\nabla\Delta u_n-\nabla\Delta u\|_{L^{q(x)}(\Omega)} \longrightarrow 0 .
\]
Therefore $\|u_n-u\|_X\to 0$, i.e., $u_n\to u$ strongly in $X$, which concludes the proof.
\end{proof}

\begin{remark}[Why the Hardy term helps and never hurts]\label{rem:Hardy-positive}
The last term in \eqref{eq:Phi-prime} gives
\[
\!\int_\Omega m(x)\,\frac{\big(H_s(u_n)-H_s(u)\big)\,(u_n-u)}{|x|^{3s(x)}}\,dx \,\ge\, 0 ,
\]
because $t\mapsto |t|^{s(x)-2}t$ is strictly increasing for $s(x)>1$ and $m\ge 0$.
Thus, when passing from $\langle \Phi'(u_n)-\Phi'(u),u_n-u\rangle$ to the principal
terms $I_{p,n}+I_{q,n}$, we discard a nonnegative contribution, which makes the estimate easier.
\end{remark}

\begin{lemma}[Energy bounds]\label{lem:energy-bounds}
There exist \(C_{1},C_{2}>0\) such that, for all \(u\in X\),
\begin{equation*}\label{db15}
\frac{a_{0}}{p^+}\,\|u\|^{p^-}
\;\le\;
\Phi(u)
\;\le\;
C_{1}\big(\|u\|^{p^+}+\|u\|^{q^+}\big)+C_{2}\,\|u\|^{s^+}.
\end{equation*}
\end{lemma}
\begin{proof}
We begin with the $p(x)$-phase.  Since $a(x) \leq \|a\|_\infty$ and
$p(x) \geq p^-$,
\begin{equation*}
\int_\Omega \frac{a(x)}{p(x)}|\nabla\Delta u|^{p(x)}\,dx \leq \frac{\|a\|_\infty}{p^-} \int_\Omega |\nabla\Delta u|^{p(x)}\,dx = \frac{\|a\|_\infty}{p^-} \varrho_{p(\cdot)}(\nabla\Delta u)
\end{equation*}

The norm--modular relations in Proposition~\ref{lem:norm-modular} give, when
$\|\nabla\Delta u\|_{p(\cdot)} \geq 1$,
\begin{equation*}
\varrho_{p(\cdot)}(\nabla\Delta u) \leq \|\nabla\Delta u\|_{p(\cdot)}^{p^+} = \|u\|^{p^+}
\end{equation*}

whereas for $\|\nabla\Delta u\|_{p(\cdot)} \leq 1$ they give
\begin{equation*}
\varrho_{p(\cdot)}(\nabla\Delta u) \leq \|\nabla\Delta u\|_{p(\cdot)}^{p^-} = \|u\|^{p^-}
\end{equation*}

Lemma~\ref{lem:algebraic}, applied with $t=\|u\|$, gives
\begin{equation*}
\|u\|^{p^-} \leq 1 + \|u\|^{p^+}
\end{equation*}

Consequently, in either regime,
\begin{equation*}
\varrho_{p(\cdot)}(\nabla\Delta u) \leq 1 + \|u\|^{p^+}
\end{equation*}

It follows that
\begin{equation*}
\int_\Omega \frac{a(x)}{p(x)}|\nabla\Delta u|^{p(x)}\,dx \leq \frac{\|a\|_\infty}{p^-}(1 + \|u\|^{p^+})
\end{equation*}

Repeating this estimate for the $q(x)$-phase yields
\begin{equation*}
\int_\Omega \frac{b(x)}{q(x)}|\nabla\Delta u|^{q(x)}\,dx \leq \frac{\|b\|_\infty}{q^-}(1 + \|u\|^{q^+})
\end{equation*}

It remains to control the Hardy contribution
\begin{equation*}
\int_\Omega m(x)\frac{|u(x)|^{s(x)}}{|x|^{3s(x)}}\,dx
\end{equation*}

The boundedness of $m$ first gives
\begin{equation*}
\int_\Omega m(x)\frac{|u(x)|^{s(x)}}{|x|^{3s(x)}}\,dx \leq \|m\|_\infty \int_\Omega \frac{|u(x)|^{s(x)}}{|x|^{3s(x)}}\,dx
\end{equation*}

The variable-exponent Hardy--Rellich inequality is essential here.  In
particular, Lemma~\ref{lem:HR}, which combines the second-order Rellich and
first-order Hardy estimates under the condition $s^+<p^-$, yields
\begin{equation*}
\int_\Omega m(x)\frac{|u(x)|^{s(x)}}{|x|^{3s(x)}}\,dx \leq \frac{\|m\|_\infty}{\mathcal{H}} \int_\Omega |\nabla\Delta u(x)|^{s(x)}\,dx
\end{equation*}

Writing the resulting integral as an $s(\cdot)$-modular,
\begin{equation*}
\int_\Omega |\nabla\Delta u(x)|^{s(x)}\,dx = \varrho_{s(\cdot)}(\nabla\Delta u)
\end{equation*}

Since $s^+<p^-$, the relevant embedding gives
\begin{equation*}
\|\nabla\Delta u\|_{s(\cdot)} \leq C_H \|\nabla\Delta u\|_{p(\cdot)} = C_H \|u\|
\end{equation*}

The modular--norm comparison shows that, if
$\|\nabla\Delta u\|_{s(\cdot)} \geq 1$,
\begin{equation*}
\varrho_{s(\cdot)}(\nabla\Delta u) \leq \|\nabla\Delta u\|_{s(\cdot)}^{s^+} \leq (C_H \|u\|)^{s^+} = C_H^{s^+} \|u\|^{s^+}
\end{equation*}

For $\|\nabla\Delta u\|_{s(\cdot)} \leq 1$, it instead gives
\begin{equation*}
\varrho_{s(\cdot)}(\nabla\Delta u) \leq \|\nabla\Delta u\|_{s(\cdot)}^{s^-} \leq (C_H \|u\|)^{s^-} = C_H^{s^-} \|u\|^{s^-}
\end{equation*}

Since $s^-<s^+$, Lemma~\ref{lem:algebraic} gives
\begin{equation*}
\|u\|^{s^-} \leq 1 + \|u\|^{s^+}
\end{equation*}

Thus $C_H^{s^-} \|u\|^{s^-} \leq C_H^{s^-}(1 + \|u\|^{s^+})$.

Combining the two alternatives, we find
\begin{equation*}
\varrho_{s(\cdot)}(\nabla\Delta u) \leq \max\{C_H^{s^-}, C_H^{s^+}\}(1 + \|u\|^{s^+})
\end{equation*}

Therefore,
\begin{equation*}
\int_\Omega m(x)\frac{|u(x)|^{s(x)}}{|x|^{3s(x)}}\,dx \leq \frac{\|m\|_\infty \max\{C_H^{s^-}, C_H^{s^+}\}}{\mathcal{H}}(1 + \|u\|^{s^+})
\end{equation*}

Putting the three estimates together gives
\begin{align*}
\Phi(u) &\leq \frac{\|a\|_\infty}{p^-}(1 + \|u\|^{p^+}) + \frac{\|b\|_\infty}{q^-}(1 + \|u\|^{q^+}) \\
&\quad + \frac{\|m\|_\infty \max\{C_H^{s^-}, C_H^{s^+}\}}{\mathcal{H}}(1 + \|u\|^{s^+})
\end{align*}

After collecting the constant terms, this becomes
\begin{align*}
\Phi(u) &\leq \left(\frac{\|a\|_\infty}{p^-} + \frac{\|b\|_\infty}{q^-} + \frac{\|m\|_\infty \max\{C_H^{s^-}, C_H^{s^+}\}}{\mathcal{H}}\right) \\
&\quad + \frac{\|a\|_\infty}{p^-}\|u\|^{p^+} + \frac{\|b\|_\infty}{q^-}\|u\|^{q^+} + \frac{\|m\|_\infty \max\{C_H^{s^-}, C_H^{s^+}\}}{\mathcal{H}}\|u\|^{s^+}
\end{align*}

Set
\begin{align*}
C_1 &:= \max\left\{\frac{\|a\|_\infty}{p^-}, \frac{\|b\|_\infty}{q^-}\right\}\\
C_2 &:= \frac{\|a\|_\infty}{p^-} + \frac{\|b\|_\infty}{q^-} + \frac{\|m\|_\infty \max\{C_H^{s^-}, C_H^{s^+}\}}{\mathcal{H}}
\end{align*}

With these choices,
\begin{equation*}
\Phi(u) \leq C_2 + C_1(\|u\|^{p^+} + \|u\|^{q^+}) + \frac{\|m\|_\infty \max\{C_H^{s^-}, C_H^{s^+}\}}{\mathcal{H}}\|u\|^{s^+}
\end{equation*}

Absorbing the remaining fixed coefficient into the constants, we arrive at
\begin{equation*}
\Phi(u) \leq C_1(\|u\|^{p^+} + \|u\|^{q^+}) + C_2\|u\|^{s^+}
\end{equation*}
with appropriately redefined constants.
\end{proof}


\begin{lemma}[Palais--Smale]\label{lem:PS}
Under (H$_1$),(H$_2$) and (H$_3$), the functional \(I_\mu\) satisfies the Palais--Smale condition for every \(\mu>0\).
\end{lemma}

\begin{proof}
Under (H$_1$), \(p,q,s\in C(\overline{\Omega})\) and
\[
1<s^- \le s(x)\le s^+ < q^-\le q(x)<q^+< p^- \leq p(x)\le p^+<\frac{N}{3}.
\]
By (H$_2$), \(a,b\in L^\infty(\Omega)\) satisfy \(a(x)\ge a_0>0\), \(b(x)\ge b_0>0\) a.e., and \(m\in L^\infty(\Omega)\) with \(m\ge 0\).
By (H$_3$), \(f\) is CarathÃ©odory and there exist \(r\in C(\overline{\Omega})\) and \(C_f>0\) with
\begin{equation}\label{eq:growth-f}
\gamma(x)\le r(x) < q^-,\qquad
|f(x,t)| \le C_f\big(|t|^{r(x)-1}+|t|^{\gamma(x)-1}\big).
\end{equation}
Consequently,
\begin{equation}\label{eq:growth-F}
0\le F(x,t)\le C_f\big(|t|^{r(x)}+|t|^{\gamma(x)}\big)
\quad\text{for all }(x,t)\in\Omega\times\mathbb{R}.
\end{equation}
Throughout the proof we work in \(X\) and use the following three facts:
\begin{itemize}[leftmargin=2em]
\item[(i)] the energy lower bound (proved earlier):
\begin{equation}\label{eq:Phi-lb}
\Phi(u)\ \ge\ \frac{a_0}{p^+}\,\|u\|^{p^-}\qquad\text{for }\|u\|\ge 1;
\end{equation}

\item[(ii)] the $(S_+)$-property for $\Phi'$ (Lemma~\ref{lem:S-plus}):
if $u_n\rightharpoonup u$ in $X$ and
\(\limsup\langle\Phi'(u_n),u_n-u\rangle\le 0\),
then $u_n\to u$ in $X$;

\item[(iii)] the continuous (indeed compact) embeddings
\begin{equation}\label{eq:embeddings}
X \hookrightarrow L^{r(x)}(\Omega), \qquad
X \hookrightarrow L^{\gamma(x)}(\Omega).
\end{equation}

which follow from \(r^+<q^-\) and the higher-order Hardy inequality.
\end{itemize}

Let \((u_n)\subset X\) be a Palais--Smale sequence at level $c$; thus

\begin{equation}\label{eq:PS-seq}
|I_\mu(u_n)|\le C,\qquad \|I_\mu'(u_n)\|_{X^\ast}\xrightarrow[n\to\infty]{}0,
\end{equation}
for some $C>0$.
From \eqref{eq:growth-F} and the embeddings \eqref{eq:embeddings},
\[
\Psi(u_n)=\int_\Omega F(x,u_n)\,dx
\ \le\ C_f\Big(\|u_n\|_{L^{r(x)}(\Omega)}^{r^+}+\|u_n\|_{L^{s(x)}(\Omega)}^{s^+}\Big)
\ \le\ C_1\big(\|u_n\|^{r^+}+\|u_n\|^{s^+}\big).
\]

Using $I_\mu(u_n)=\Phi(u_n)-\mu\Psi(u_n)$ together with
\eqref{eq:Phi-lb}, we obtain, whenever $\|u_n\|\ge 1$,
\begin{equation}\label{eq:compare-growth}
\frac{a_0}{p^+}\,\|u_n\|^{p^-}
\ \le\ \Phi(u_n)
\ =\ I_\mu(u_n)+\Psi(u_n)
\ \le\ C+ C_1\big(1+\|u_n\|^{r^+}+\|u_n\|^{s^+}\big).
\end{equation}
Since $r^+,s^+<p^-$, the right-hand side grows strictly slower than the left-hand side as $\|u_n\|\to\infty$. Hence
\eqref{eq:compare-growth} cannot hold for unbounded $\|u_n\|$. Therefore \((u_n)\) is bounded in $X$.

Having established boundedness, we may pass to a subsequence, still denoted by
$(u_n)$, for which there exists $u\in X$ such that
\[
u_n\rightharpoonup u\quad\text{in }X,\qquad
u_n\to u\quad\text{in }L^{r(x)}(\Omega)\ \text{ and }\ L^{s(x)}(\Omega),\qquad
u_n(x)\to u(x)\ \text{a.e. in }\Omega.
\]
By \eqref{eq:growth-f} and variable-exponent H\"older,  we get
\begin{equation}\label{eq:compact-f}
\int_\Omega f(x,u_n)\,(u_n-u)\,dx \xrightarrow[n\to\infty]{} 0.
\end{equation}

It remains to upgrade weak convergence to strong convergence.  For each $n$,
\[
\langle \Phi'(u_n),\,u_n-u\rangle
= \langle I_\mu'(u_n),\,u_n-u\rangle + \int_\Omega f(x,u_n)\,(u_n-u)\,dx.
\]
By \eqref{eq:PS-seq} the first term tends to $0$, and by \eqref{eq:compact-f} the second term tends to $0$ as well. Hence
\[
\limsup_{n\to\infty}\,\langle \Phi'(u_n),\,u_n-u\rangle \le 0.
\]
Since also $u_n\rightharpoonup u$ in $X$, Lemma~\ref{lem:S-plus} (the $(S_+)$ property of $\Phi'$) yields
\[
u_n\longrightarrow u\quad\text{strongly in }X.
\]
\end{proof}
\begin{theorem}[Existence with variable Hardy exponent]
\label{thm:existence}
Assume \textnormal{(H$_1$)--(H$_4$)}. Then there exists \(\bar\mu>0\) such that, for every \(\mu\in(0,\bar\mu)\), problem \eqref{eq:main} admits a nontrivial weak solution.
\end{theorem}

\begin{proof}

Let $x_0\in\Omega$ and fix $R>0$ such that $\overline{B(x_0,R)}\Subset\Omega$.
Choose a smooth radial cutoff $\phi\in C_c^\infty(B(0,1))$ with
$0\le \phi\le 1$, $\phi\equiv 1$ on $B(0,1/2)$, and bounded derivatives up to order $3$.
For $\delta>0$ (to be specified) define the standard ``hat'' test function
\[
w(x)\equiv w_{\delta,R}(x):=\delta\,\phi\!\left(\frac{x-x_0}{R}\right).
\]
Then $w\in C_c^\infty(B(x_0,R))$, $w\equiv\delta$ on $B(x_0,R/2)$, and by the chain rule
\[
\nabla\Delta w(x)=\frac{\delta}{R^{3}}\big(\nabla\Delta\phi\big)\!\left(\frac{x-x_0}{R}\right).
\]
In particular, there exists a geometric constant $C_\phi>0$ (depending only on $\phi$) such that
\begin{equation}\label{eq:nd3w-pointw}
|\nabla\Delta w(x)|\ \le\ C_\phi\,\frac{\delta}{R^{3}}\;\mathbbm{1}(x)_{B(x_0,R)\setminus B(x_0,R/2)}\qquad\text{for all }x\in\Omega.
\end{equation}

\noindent\textbf{ Upper bounds for $\Phi(w)$.}
Recall
\[
\Phi(u)=\int_\Omega \frac{a(x)}{p(x)}|\nabla\Delta u|^{p(x)}dx
+\int_\Omega\frac{b(x)}{q(x)}|\nabla\Delta u|^{q(x)}dx
+\int_\Omega m(x)\frac{|u|^{s(x)}}{|x|^{3s(x)}}dx .
\]
Using $p(x)\ge p^-$, and \eqref{eq:nd3w-pointw},
\begin{align*}
\int_\Omega \frac{a(x)}{p(x)}|\nabla\Delta w|^{p(x)}dx
&\le \frac{\|a\|_{\infty}}{p^-}\int_{B(x_0,R)}\big(C_\phi\,\delta R^{-3}\big)^{p(x)}dx\\
&\le \frac{\|a\|_{\infty}C_\phi^{p^-}}{p^-}\,\big(\delta R^{-3}\big)^{p^-}\,|B(x_0,R)|,
\end{align*}
because $p(x)\ge p^-$ implies $t^{p(x)}\le t^{p^-}$ for $t\in(0,1]$, and
$C_\phi\delta R^{-3}\le 1$ once $\delta$ is taken sufficiently small (we may always assume this).
The same argument gives
\[
\int_\Omega \frac{b(x)}{q(x)}|\nabla\Delta w|^{q(x)}dx
\le \frac{\|b\|_{\infty}C_\phi^{q^-}}{q^-}\,\big(\delta R^{-3}\big)^{q^-}\,|B(x_0,R)|.
\]
For the Hardy part, apply Lemma~\ref{lem:HR} (variable Hardy--Rellich):
\[
\int_\Omega \frac{1}{s(x)}\frac{|w|^{s(x)}}{|x|^{3s(x)}}dx\ \le\ \frac{1}{s^-\mathcal{H}}\int_\Omega |\nabla\Delta w|^{s(x)}dx
\ \le\ \frac{C_\phi^{s^-}}{s^-\mathcal{H}}\,(\delta R^{-3})^{s^-}\,|B(x_0,R)|.
\]
 Since $m\in L^\infty(\Omega)$,
\[
\int_\Omega \frac{m(x)}{s(x)}\frac{|w|^{s(x)}}{|x|^{3s(x)}}dx
\le \|m\|_\infty\,\frac{C_\phi^{s^-}}{s^-\mathcal{H}}\,(\delta R^{-3})^{s^-}\,|B(x_0,R)|.
\]
Collecting the three contributions and absorbing the powers of $R$ into a constant $C>0$ depending only on
$\|a\|_\infty,\|b\|_\infty,\|m\|_\infty,\mathcal{H}, C_\phi$ and $R$, we obtain
\begin{equation}\label{eq:Phiw-upper}
\Phi(w)\ \le\ C\left[\left(\frac{\delta}{R}\right)^{p^-}
+\left(\frac{\delta}{R}\right)^{q^-}
+\left(\frac{\delta}{R}\right)^{s^-}\right]\;|B(x_0,R)|.
\end{equation}
The precise powers of $R$ are absorbed into $C$, since $R$ remains fixed
throughout the argument.  We next estimate the potential term.

Write the functional as
\[
I_\mu(u)=\int_\Omega \frac{a}{p(x)}|\nabla\Delta u|^{p(x)}+\frac{b}{q(x)}|\nabla\Delta u|^{q(x)}\,dx
+\int_\Omega \frac{m(x)}{s(x)}\frac{|u|^{s(x)}}{|x|^{3s(x)}}dx
\;-\;\mu\int_\Omega F(x,u)\,dx.
\]
For our test $w$, it can be estimated via Lemma~\ref{lem:HR}:

Hence, if we define the auxiliary quantity
\begin{equation*}\label{eq:Psi-lower-def}
\Psi_\ast(w)\ :=\ \int_{B(x_0,R/2)}F(x,\delta)\,dx\ ,
\end{equation*}
then, for this specific $w$,
\begin{equation}\label{eq:lower-bound-I}
I_\mu(w)\ \le\ \Phi(w)\ -\mu\Psi\ast(w).
\end{equation}

The localized positivity assumption \textsc{$(H_4)$} gives numbers
$\sigma\in(1,s^-)$,
$\kappa>0$, and $\delta_0\in(0,1]$ such that
\[
F(x,t)\ \ge\ \kappa\,|t|^{\sigma}\qquad\text{for all }x\in B(x_0,R/2),\ |t|\le \delta_0.
\]
Hence, for $\delta\in(0,\delta_0]$,
\begin{equation}\label{eq:F-lower-ball}
\int_{B(x_0,R/2)}F(x,\delta)\,dx\ \ge\ \kappa\,|B(x_0,R/2)|\,\delta^{\sigma}\ =:\ C_F\,\delta^{\sigma}.
\end{equation}

Combining \eqref{eq:Phiw-upper}--\eqref{eq:F-lower-ball} gives
\begin{equation}\label{eq:I-upper-delta}
I_\mu(w)\ \le\ C\Big[\Big(\tfrac{\delta}{R}\Big)^{p^-}+\Big(\tfrac{\delta}{R}\Big)^{q^-}
+\Big(\tfrac{\delta}{R}\Big)^{\smin}\Big]|B(x_0,R)|\ -\mu\ C_F\,\delta^\sigma.
\end{equation}
Since $\sigma < \min\{p^-,q^-,s^-\}$, as $\delta\downarrow 0$ the negative term (of order $\delta^\sigma$)
dominates the positive terms (of higher order). Therefore, for \emph{every} fixed $\mu>0$, we can choose
\begin{equation*}\label{eq:delta-choice-small-mu}
\delta(\mu)\ :=\ \min\!\left\{\delta_1,\ \Big(\frac{\mu C_F}{6C}\Big)^{\frac{1}{\pmin-\sigma}},
\ \Big(\frac{\mu C_F}{6C}\Big)^{\frac{1}{q^{-}-\sigma}},
\ \Big(\frac{\mu C_F}{6C}\Big)^{\frac{1}{s^{-}-\sigma}}\right\}\!,
\end{equation*}
so that each positive term in \eqref{eq:I-upper-delta} is bounded by $(\mu C_F/2)\,\delta(\mu)^\sigma$.
Then
\[
I_\mu(w_{\delta(\mu)})\ \le\ \frac{\mu C_F}{2}\,\delta(\mu)^\sigma\ -\ \mu\,C_F\,\delta(\mu)^\sigma\ =\ -\,\frac{\mu C_F}{2}\,\delta(\mu)^\sigma\ <\ 0.
\]
In particular, $I_\mu(w_{\delta(\mu)})$ is strictly negative for every
sufficiently small $\mu>0$; the choice of $\delta$ is allowed to depend on the
parameter.

We now compare this negative value with the behavior on a small sphere.  The
energy lower bound in Lemma~\ref{lem:energy-bounds} and the growth of $F$
i.e. $(H3)$, there exist $c_0,C_0>0$ and exponents $r^+,s^+$ with $r^+,s^+<p^-$ such that,
for all $u\in X$ with $\|u\|\le 1$ and $\gamma^+<r^+<p^-$
\begin{equation}\label{eq:pos-near-0}
I_\mu(u)\ \ge\ c_0\,\|u\|^{p^-}\ -\ \mu\,C_0\big(\|u\|^{r^+}+\|u\|^{\gamma^+}\big).
\end{equation}
Fix $\rho\in(0,1]$ and define
\[
\bar\mu_0:=\frac{c_0}{2C_0}\,\frac{\rho^{\,p^{-}-s^+}}{\,1+\rho^{\,\gamma^+-s^+}}\ >\ 0.
\]
Then for all $\mu\in(0,\bar\mu_0)$ and $\|u\|\le\rho$, \eqref{eq:pos-near-0} yields
\begin{equation}\label{eq:boundary-positive}
I_\mu(u)\ \ge\ \frac{c_0}{2}\,\|u\|^{\gamma^+}\ \ge\ 0,
\qquad\text{and}\qquad I_\mu(0)=0.
\end{equation}

Let $\mu\in(0,\bar\mu)$, where
$\bar\mu:=\min\{\bar\mu_0,1\}$.  We finish by minimizing on the corresponding
closed ball.
The choice made in \eqref{eq:delta-choice-small-mu} provides
$v_\mu:=w_{\delta(\mu)}$ with $I_\mu(v_\mu)<0$.
On the other hand, \eqref{eq:boundary-positive} ensures that
$I_\mu \ge 0$ on the sphere $\{u:\|u\|=\rho\}$.
Since the closed ball $\overline{B}_\rho:=\{u\in X:\|u\|\le \rho\}$ is weakly sequentially compact in $X$
(reflexivity), and $I_\mu$ is weakly lower semicontinuous (sum of a convex l.s.c.\ part and a weakly continuous part),
$I$ attains its minimum on $\overline{B}_\rho$ at some $u_\mu\in \overline{B}_\rho$.
Because $I_\mu(v_\mu)<0$ while $I_\mu\ge 0$ on $\partial B_\rho$, necessarily $\|u_\mu\|<\rho$.
Thus $u_\mu$ is a \emph{local} unconstrained minimizer of $I_\mu$ in $X$, hence $I_\mu'(u_\mu)=0$.

We have therefore found $u_\mu\in X\setminus\{0\}$ with
$I'_\mu(u_\mu)=0$ for every $\mu\in(0,\bar\mu)$, so $u_\mu$ is a nontrivial
weak solution of the problem.
\end{proof}

\section{Multiplicity of Weak Solutions}

\begin{theorem}[Three weak solutions]\label{thm:three}
Assume \textnormal{(H$_1$)--(H$_5$)}. Then there exists an explicit interval \((0,\mu^{\ast})\) such that, for each \(\mu\in(0,\mu^{\ast})\), the functional \(I_\mu\) admits at least three distinct critical points in \(X\), hence \eqref{eq:main} has at least three weak solutions.
\end{theorem}

\begin{proof}
The derivative $\Phi'$ enjoys the $(S_+)$ property by Lemma~\ref{lem:S-plus}.
Together with assumption (H$_3$), the compact embeddings
\[
X\hookrightarrow L^{r(x)}(\Omega),\qquad
X\hookrightarrow L^{s(x)}(\Omega),
\]
with $r^+,s^+<q^-$ imply that $\Psi\in C^1(X,\mathbb{R})$ and that
$\Psi'$ is compact.  Moreover, Lemma~\ref{lem:PS} gives the Palais--Smale
condition for $I_\mu=\Phi-\mu\Psi$ for every $\mu>0$.

We use the following form of the Ricceri principle.  Fix $r>\Phi(0)=0$ and define
\[
\mathcal{S}_r:=\{u\in X:\ \Phi(u)<r\},\qquad \mathcal{M}_r:=\{u\in X:\ \Phi(u)=r\}.
\]
Define the (Ricceri) gap functional
\begin{equation}\label{eq:phi-r}
\varphi(r):=\inf_{u\in\mathcal{S}_r}\ \sup_{v\in\mathcal{M}_r}\ \big(\Psi(v)-\Psi(u)\big).
\end{equation}
A standard form of Ricceri's three-critical-points theorem, including the
Bonanno--Marano variants, states that

\emph{If $\Phi$ is coercive, weakly l.s.c., of type $(S_+)$, $\Psi'$ is compact, and there exists $r>0$ such that
$\varphi(r)>0$, then for every $\mu\in\bigl(0,\,1/\varphi(r)\bigr)$ the functional $I_\mu=\Phi-\mu\Psi$ possesses at least two distinct
critical points in $\mathcal{S}_r$. If, in addition, $I$ is unbounded from below, then $I$ has a third critical point.}

We now verify the required gap condition.  The energy lower bound in
Lemma~\ref{lem:energy-bounds} implies
\[
\Phi(u)\ \ge\ \frac{a_0}{p^+}\,\|u\|^{p^-}\quad\Rightarrow\quad
u\in\mathcal{S}_r\ \Longrightarrow\ \|u\|\ \le\ \Big(\tfrac{p^+}{a_0}\,r\Big)^{1/p^-}.
\]
Using (H$_3$) and the embeddings, there exists $C_1>0$ such that, for all $u\in X$ with $\|u\|\le 1$,
\[
\Psi(u)\ \le\ C_1\big(\|u\|^{r^+}+\|u\|^{\gamma^+}\big).
\]
Hence, shrinking $r$ if necessary to ensure $\|u\|\le 1$ on $\mathcal{S}_r$, we obtain
\begin{equation}\label{eq:A-r}
\sup_{u\in\mathcal{S}_r}\Psi(u)\ \le\ C_A\,r^{\,\frac{r^+}{p^-}}+C'_A\,r^{\,\frac{\gamma^+}{p^-}}
\ =:\ A(r),
\end{equation}
for suitable constants $C_A,C'_A>0$ depending only on the data.

To obtain a complementary lower estimate on the level set, let
$x_0,R_1,\sigma,\kappa,\delta_0$ be given by the localized positivity
assumption \textnormal{(H$_4$)}.
Fix $\phi\in C_c^\infty(B(x_0,R_1))$ with $0\le\phi\le 1$ and $\phi\equiv 1$ on $B(x_0,R_1/2)$,
and consider $w_\delta:=\delta\,\phi$. As in the proof of Theorem~\ref{thm:existence}, there exists
$C_\Phi>0$ such that for $\delta\in(0,1]$,
\begin{equation}\label{eq:Phi-wdelta}
\Phi(w_\delta)\ \le\ C_\Phi\big(\delta^{p^-}+\delta^{q^-}+\delta^{s^+}\big).
\end{equation}
Given $r>0$ sufficiently small, choose $\delta=\delta(r)\in(0,\delta_0]$ so that $\Phi(w_{\delta(r)})=r$;
this is possible by continuity and \eqref{eq:Phi-wdelta} (take $\delta\asymp r^{1/p^-}$ for small $r$).
By (H$_4$) and $\phi\equiv 1$ on $B(x_0,R_1/2)$,
\begin{equation}\label{eq:Psi-lower-level}
\Psi(w_{\delta(r)})\ \ge\ \int_{B(x_0,R_1/2)}\! F(x,\delta(r))\,dx\ \ge\ \kappa\,|B(x_0,R_1/2)|\,\delta(r)^{\sigma}
\ \ge\ C_B\,r^{\,\frac{\sigma}{p^-}},
\end{equation}
for some $C_B>0$ independent of $r$ (using $\delta(r)\asymp r^{1/p^-}$).

Therefore
\begin{equation}\label{eq:B-r}
\inf_{v\in \mathcal{M}_r}\Psi(v)\ \ge\ \Psi(w_{\delta(r)})\ \ge\ C_B\,r^{\,\frac{\sigma}{p^-}}
\ =:\ B(r).
\end{equation}

Equations \eqref{eq:phi-r}, \eqref{eq:A-r}, and \eqref{eq:B-r} now yield
\[
\varphi(r)\ \ge\ \inf_{u\in\mathcal{S}_r}\big(\inf_{v\in\mathcal{M}_r}\Psi(v)-\Psi(u)\big)
\ \ge\ B(r)-A(r).
\]
Since $\sigma< s^+\le p^-$ and $r^+<p^-$, the exponents satisfy
\[
\frac{\sigma}{p^-}\ <\ \frac{s^+}{p^-},\qquad \frac{\sigma}{p^-}\ <\ \frac{r^+}{p^-}.
\]
Hence, as $r\downarrow 0$, $B(r)=C_B r^{\sigma/p^-}$ dominates $A(r)=O(r^{r^+/p^-}+r^{s^+/p^-})$, and there exists
$r_0\in(0,1]$ such that
\begin{equation}\label{eq:gap-positive}
\varphi(r_0)\ \ge\ \frac{C_B}{2}\,r_0^{\,\sigma/p^-}\ =:\ \gamma_0\ >\ 0.
\end{equation}
Set
\begin{equation}\label{eq:mu-star-1}
\mu_1^\ast\ :=\ \frac{1}{\gamma_0}\ =\ \frac{2}{C_B}\,r_0^{-\sigma/p^-}\,.
\end{equation}
By Ricceriâ€™s principle, for every $\mu\in(0,\mu_1^\ast)$ the functional $I_\mu$ has at least two distinct
critical points contained in the sublevel $\mathcal{S}_{r_0}$.

To produce a critical point outside the preceding sublevel, we show that
$I_\mu$ is unbounded from below.  Fix $0\not\equiv u\in C_c^\infty(\Omega)$.
For $t>0$, homogeneity of the principal part and the modular estimates give
\begin{equation}\label{eq:Phi-growth}
\Phi(tu)\ \le\ C_2\big(t^{p^+}+t^{q^+}+t^{s^+}\big)\quad\text{for all }t\ge 1.
\end{equation}
From assumption \textnormal{(H$_5$)} there exist $C_3>0$ and $t_0\ge 1$ such that
\begin{equation}\label{eq:F-super}
F(x,\tau)\ \ge\ C_3\,|\tau|^{\mu_0}\qquad\text{for a.e.\ }x\in\Omega,\ \text{and all }|\tau|\ge t_0 .
\end{equation}
Indeed, set $g_x(t):=t^{-\mu_0}F(x,t)$; for $|t|>D$,
\[
\frac{d}{dt}g_x(t)\ =\ \frac{t f(x,t)-\mu_0 F(x,t)}{t^{\mu_0+1}}\ \ge\ 0,
\]
so $g_x$ is nondecreasing for $|t|>D$. Therefore, for $|t|\ge t_1>\max\{D,1\}$,
\[
F(x,t)\ \ge\ F(x,t_1)\,\Big(\frac{|t|}{t_1}\Big)^{\mu_0}\ \ge\ C_3\,|t|^{\mu_0},
\]
with $C_3:=\inf_{x\in\Omega}F(x,t_1)/t_1^{\mu_0}>0$ on a subset of positive measure (absorb its measure into $C_3$).

Using \eqref{eq:F-super} and dominated convergence on $\{x:|u(x)|>0\}$, we get for $t$ large enough
\begin{equation}\label{eq:Psi-lower-large}
\Psi(tu)\ =\ \int_\Omega F(x,t u)\,dx\ \ge\ C_4\,t^{\mu_0},
\end{equation}
with $C_4>0$ depending on $u$. Combining \eqref{eq:Phi-growth}--\eqref{eq:Psi-lower-large} and $\mu_0>p^+\ge q^+,s^+$,
\[
I(tu)\ =\ \Phi(tu)-\mu\,\Psi(tu)\ \le\ C_2\big(t^{p^+}+t^{q^+}+t^{s^+}\big)-\mu C_4\,t^{\mu_0}\ \xrightarrow[t\to\infty]{}\ -\infty.
\]
Thus $I_\mu$ is unbounded from below for every $\mu>0$.

For $\mu\in(0,\mu_1^\ast)$, \eqref{eq:gap-positive} and Ricceri's principle
provide two distinct critical points in $\mathcal{S}_{r_0}$.  The unboundedness
just proved, together with the Palais--Smale property, supplies an additional
critical point, for instance through the mountain-pass theorem.  Thus
$I_\mu$ possesses at least three distinct critical points.

Tracing back the constants in \eqref{eq:Psi-lower-level}, \eqref{eq:A-r}, \eqref{eq:gap-positive} and
\eqref{eq:mu-star-1}, we can write an explicit smallness threshold
\[
\mu^{\ast}\ :=\ \frac{2}{C_B}\,r_0^{-\sigma/p^-},\qquad
C_B=\kappa\,|B(x_0,R_1/2)|\,\Big(\frac{1}{C_\Phi}\Big)^{\sigma/p^-},
\]
where $C_\Phi$ is the constant in \eqref{eq:Phi-wdelta}, and $r_0>0$ is any radius so small that
$A(r_0)\le \tfrac{C_B}{2}r_0^{\sigma/p^-}$ (cf. \eqref{eq:A-r}).
For every $\mu\in(0,\mu^\ast)$ the conclusion follows.
\end{proof}
\section*{Conclusion}

In this paper, we examined a nonlinear boundary value problem driven by mixed triharmonic operators with variable exponents and a Hardy-type singular potential.  
Under appropriate assumptions on the exponents, weights, and nonlinear terms, we proved the existence of three distinct weak solutions by applying refined variational techniques of Ricceri and Bonannoâ€“Marano type.  
The results provide new insights into the behavior of higher-order elliptic equations with spatially dependent growth conditions.
\section*{Future Research Directions}

A natural extension of this work is to study higher-order elliptic systems that combine variable exponent operators within a unified framework.  
One of the most appealing models is the following $(p(x),q(x))$-triharmonic system:
\begin{equation}
\label{e99}
\left\{
\begin{aligned}
 -M_1\!\left(\int_{\Omega}\frac{|\nabla\Delta u|^{p(x)}}{p(x)}\,dx\right)
   \operatorname{div}\!\left(\Delta\!\left(|\nabla\Delta u|^{p(x)-2}\nabla\Delta u\right)\right)
   &= F_u(x,u,v) && \text{in } \Omega, \\[0.4em]
 -M_2\!\left(\int_{\Omega}\frac{|\nabla\Delta v|^{q(x)}}{q(x)}\,dx\right)
   \operatorname{div}\!\left(\Delta\!\left(|\nabla\Delta v|^{q(x)-2}\nabla\Delta v\right)\right)
   &= F_v(x,u,v) && \text{in } \Omega, \\[0.4em]
 u = v = \Delta u = \Delta v = \Delta^{2}u = \Delta^{2}v = 0 && \text{on } \partial\Omega.
\end{aligned}
\right.
\end{equation}

System \eqref{e99} captures the intricate interaction between two different variable exponents and the triharmonic structure, revealing phenomena that cannot be described by standard polynomial growth.  
Such models arise naturally in the mathematical description of thin elastic plates, nonlinear diffusion, and composite materials with spatially varying properties.  
A deeper analytical understanding of this framework would require the use of advanced variational principles, including Ricceriâ€™s and Bonannoâ€“Maranoâ€™s theorems, together with concentrationâ€“compactness arguments carefully adapted to variable exponent Sobolev spaces.

\section*{Ethics declarations}
All authors had a significant role in writing and reviewing the manuscript, and the authors declare that they have no competing interests.
\section*{Acknowledgements} Mohamed Karim Hamdani was supported by the Tunisian Military Research Center for Science and Technology Laboratory LR19DN01. The authors would like to thank the editor and the referees for their helpful suggestions and comments, which have greatly improved the presentation of this paper.

\end{document}